\documentclass[12pt,a4paper]{amsart}
\usepackage{amsopn,amssymb,mathrsfs}
 \allowdisplaybreaks

\newtheorem{thm}{Theorem}[section]

\newtheorem{lem}[thm]{Lemma}
\newtheorem{prop}[thm]{Proposition}
\newtheorem{prob}[thm]{Problem}
\newtheorem{conj}[thm]{Conjecture}

\theoremstyle{definition}

\numberwithin{equation}{section}

\newcommand{\czero}{\ensuremath{c_0}}

\newcommand{\dual}[1]{\ensuremath{{#1}^*}}

\newcommand{\Gdelta}{\ensuremath{G_\delta}}

\newcommand{\ip}[2]{\ensuremath{\langle{#1}\,,\,{#2}\rangle}}

\newcommand{\lint}[4]{\ensuremath{\int_{#1}^{#2}{#3}\:\mathrm{d}{#4}}}

\newcommand{\lp}[1]{\ensuremath{\ell_{#1}}}

\newcommand{\mapping}[3]{\ensuremath{{#1}:{#2}\longrightarrow{#3}}}

\newcommand{\mi}{\ensuremath{\mathrm{i}}}

\newcommand{\nat}{\mathbb{N}}
\newcommand{\norm}[1]{\ensuremath{\left|\left|{#1}\right|\right|}}
\newcommand{\normdot}{\ensuremath{||\cdot||}}

\newcommand{\oneton}[2]{\ensuremath{{#1}_1,\ldots,{#1}_{#2}}}
\newcommand{\orb}[2]{\ensuremath{\orbstart({#1},{#2})}}

\newcommand{\pnorm}[2]{\ensuremath{||{#1}||_{#2}}}
\newcommand{\pnormdot}[1]{\ensuremath{||\cdot||_{#1}}}

\newcommand{\real}{\mathbb{R}}

\newcommand{\setcomp}[2]{\ensuremath{\{{#1}\;:\;\,{#2}\}}}

\newcommand{\sph}[1]{\ensuremath{S_{#1}}}

\DeclareMathOperator{\aspan}{span} %
\DeclareMathOperator{\orbstart}{orb} %

\begin{document}
\title{Operator machines on directed graphs}
\begin{abstract}
We show that if an infinite-dimensional Banach space $X$ has a
symmetric basis then there exists a bounded, linear operator
$\mapping{R}{X}{X}$ such that the set
\[
A = \setcomp{x \in X}{\norm{R^n(x)} \rightarrow \infty}
\]
is non-empty and nowhere dense in $X$. Moreover, if $x \in
X\setminus A$ then some subsequence of $(R^n(x))_{n=1}^\infty$
converges weakly to $x$. This answers in the negative a recent
conjecture of Pr\v{a}jitur\v{a}. The result can be extended to any
Banach space containing an infinite-dimensional, complemented
subspace with a symmetric basis; in particular, all `classical'
Banach spaces admit such an operator.
\end{abstract}

\author{Petr H\'{a}jek}
\address{Institute of Mathematics of the AS CR\\ \v{Z}itn\'{a} 25\\ CZ - 115 67 Praha 1\\ Czech Republic}
\email{hajek@math.cas.cz}
\author{Richard J.\ Smith}
\email{smith@math.cas.cz}
\date{\today}
\thanks{Both authors are supported by Grant A 100190801 and Institutional Research Plan
AV0Z10190503}%
\subjclass[2000]{47A05}%
\keywords{Orbits of operators}

\maketitle

\section{Introduction}

Given a Banach space $X$, a bounded linear operator $T$ on $X$ and
$x \in X$, we say that the {\em orbit of }$x$ {\em with respect to}
$T$ is the set
$$
\orb{x}{T} \;=\; \setcomp{T^n(x)}{n \geq 0}.
$$
It is well known (cf \cite{r:69}) that if $X$ is finite-dimensional
then $\orb{x}{T}$ is `regular', in the sense that either
$\norm{T^n(x)} \rightarrow 0$, $\norm{T^n(x)} \rightarrow \infty$,
or there exists $M > 0$ such that $M^{-1} \leq \norm{T^n(x)} \leq M$
for all $n$. The infinite-dimensional situation is very different.
Rolewicz provided simple examples of operators on
infinite-dimensional spaces which admit {\em hypercyclic} vectors,
that is, vectors with norm-dense orbits.

The study of orbits is connected to the invariant subspace problem.
Indeed, an operator $T$ on a Banach space $X$ has no non-trivial,
closed, invariant subspaces if and only if $\aspan \orb{x}{T}$ is
norm-dense for every non-zero $x \in X$.

There is a considerable body of literature on operators which admit
hypercyclic vectors. In this note, we study operators with more
regular orbits, in particular, those which tend to infinity. This
type of orbit has received attention from several authors. For
example, in a systematic study of orbits of operators on Hilbert
space, Beauzamy provides several sufficient conditions for $T$ to
admit a norm-dense set of points $x$ satisfying $\norm{T^n(x)}
\rightarrow \infty$ \cite[Chapter III]{beauzamy:88}. Broadly
speaking, these conditions are based on the growth of the sequence
$(\norm{T^n})_{n=1}^\infty$. For example, if $\sum_{n=1}^\infty
\norm{T^n}^{-1} < \infty$ then $T$ admits such a dense set. Sharp
estimates of this nature applying to general Banach spaces are given
in \cite{mv:09}. We refer the reader to the surveys
\cite{muller:01,muller:04} for additional results on this topic.

A given operator can have both regular and highly irregular orbits,
and the exact behaviour of $\orb{x}{T}$, as $x$ ranges over $X$, is
not so easy to determine. In \cite{p:09}, Pr\v{a}jitur\v{a} makes
the following conjecture.

\begin{conj}[{\cite[Conjecture 2.9]{p:09}}]\label{praconj}
Let $T$ be an operator $T$ on a Banach space and let
$$
A_T \;=\; \setcomp{x \in X}{\norm{T^n(x)} \rightarrow \infty}.
$$
Then $A_T$ is norm-dense whenever $A_T$ is non-empty.
\end{conj}

Of course, if $\norm{T^n(x)} \rightarrow \infty$ for some $x$ then
$(\norm{T^n})_{n=1}^\infty$ is unbounded, so by the uniform
boundedness principle, the set of $y$ with the property that $\sup_n
\norm{T^n(y)} = \infty$ is a norm-dense $\Gdelta$ in $X$. However,
this clearly does not say anything about whether $\norm{T^n(y)}$
tends to infinity or not. Indeed, the weighted backwards shift
operator $T$ on $\lp{p}$, $1 \leq p < \infty$, given by
$$
T(e_i) = \left\{
\begin{array} {l@{\quad}l}
(i/(i-1))^{\frac{1}{p}}e_{i-1} & \mbox{if } i > 1\\
0 & \mbox{if }i=1\\
\end{array} \right.
$$
satisfies $\norm{T^n} \rightarrow \infty$, but $\norm{T^n(x)}
\not\rightarrow \infty$ for all $x$ \cite[Example 4]{mv:09}. In
\cite[pp. 66--68]{beauzamy:88}, there is an example of an operator
$T$ on Hilbert space satisfying $\norm{T^n} \rightarrow \infty$, but
$\inf_n \norm{T^n(x)} = 0$ for all $x$.

The object of this note is to show that there is a wide class of
Banach spaces which admit operators failing Conjecture
\ref{praconj}. In fact, by constructing a range of suitable
operators, we can impose a reasonable degree of control over the
structure of $A_T$. Clearly, $A_T$ is always radial, in the sense
that if $x \in A_T$ and $\lambda \neq 0$ then $\lambda x \in A_T$.
Thus we need only trouble ourselves with what happens to points in
the unit sphere $\sph{X}$.

We shall consider both real and complex Banach spaces. Recall that
if a Schauder basis $(e_i)_{i=1}^\infty$ of $X$ is symmetric then
there exists an equivalent norm $\normdot$ on $X$ with the property
that
$$
\norm{\sum_{i=1}^\infty x_i e_i} \;=\; \norm{\sum_{i=1}^\infty
\lambda_i x_i e_{\pi(i)}}
$$
whenever $\pi$ is a permutation of $\nat$ and $|\lambda_i| = 1$ for
all $i$. Such a norm is called {\em symmetric}; hereafter, whenever
we have a Banach space $X$ with a symmetric basis, we shall assume
that $X$ is infinite-dimensional and that the associated norm is
symmetric. In addition, we shall say that a subset $E$ of a Banach
space is {\em symmetric} if $\lambda x \in E$ whenever $x \in E$ and
$|\lambda| = 1$. Here follows our main result.

\begin{thm}
\label{master} Let $X$ have a symmetric basis with norm $\normdot$
and suppose that $Y \subseteq X$ is a subspace of dimension $d$,
where $2 \leq d < \infty$. Moreover, let $E \subseteq \sph{Y}$ be
closed and symmetric, and let $J$ be a projection of $X$ onto $Y$.
Then there exists an operator $\mapping{R}{X}{X}$ with two
properties:
\begin{enumerate}
\item if $J(x) \in E$ then $\norm{R^n(x)} \rightarrow \infty$;
\item if $J(x) \in \sph{Y}\setminus E$ then there is a subsequence
$(R^{n_i}(x))$ of $(R^n(x))$ such that $R^{n_i}(x) \rightarrow x$
weakly.
\end{enumerate}
\end{thm}

Of course, we obtain the claim in the abstract by ensuring that $E
\subseteq \sph{Y}$ in Theorem \ref{master} is non-empty and nowhere
dense. Roughly speaking, we use the extra dimensions in the
complement of $Y$ in $X$ to encode the non-linear information in
$E$. To give an idea of what we mean by this, we can compare, at a
distance, this encoding of non-linear information to the standard
method of producing an operator on Hilbert space with a prescribed
spectrum, namely, by arranging a suitable, countable family of
eigenvalues. The proof of Theorem \ref{master} is spread across
sections \ref{sectionlocalestimates} and \ref{sectionmainproof}. We
expect that it is possible to generalise Theorem \ref{master} to
incorporate subsets of $\sph{Y}$ of greater topological complexity,
but to do so would go beyond the immediate aims of this paper and
would unduly complicate our existing proof.

Of course, if $X$ in Theorem \ref{master} is complemented in some
overspace $Z$ then we obtain a corresponding result about $Z$. In
particular, by considering $\czero$ or $\lp{p}$, $1 \leq p <
\infty$, we can see that any `classical' Banach space admits an
operator $T$ such that $A_T$ is non-empty and nowhere dense.

We finish this section by showing that the richness of structure of
$A_T$, demonstrated by Theorem \ref{master}, cannot be reproduced in
the finite-dimensional setting. In fact, we show that the operator
$R$ in Theorem \ref{master} cannot be compact in general.

\begin{prop}\label{compact}
Let $X$ be a Banach space and $\mapping{T}{X}{X}$ a compact
operator. Then there exist finite-codimensional subspaces $Z
\subseteq Y \subseteq X$ with the property that
\begin{enumerate}
\item if $x \in X\setminus Y$ then $\norm{T^n(x)} \rightarrow
\infty$;
\item if $x \in Y\setminus Z$ then there exists $M > 0$ such that
$M^{-1} \leq \norm{T^n(x)} \leq M$ for every $n$;
\item if $x \in Z$ then $\norm{T^n(x)} \rightarrow 0$.
\end{enumerate}
\end{prop}

In particular, if $T$ is compact then $A_T$ is simply the complement
of a finite-codimensional subspace and, in particular, either empty
or dense. Proposition \ref{compact} is a mild elaboration and
generalisation of the behaviour of orbits in finite-dimensional
space, stated at the beginning of this note. The result is probably
folklore but we sketch a proof of it for completeness.

\begin{proof}[Sketch proof of Proposition \ref{compact}.]
First, we assume that $X$ is complex. Using the standard spectral
theory of compact operators and the theory of Jordan normal forms,
we know that we can find a (possibly empty) sequence of
finite-dimensional subspaces $\oneton{X}{n}$, and a
finite-codimensional subspace $Z$, all invariant for $T$, such that
$$
X \;=\; X_1 \oplus \ldots \oplus X_n \oplus Z.
$$
Moreover, the subspaces $X_i$ satisfy the following properties.
\begin{enumerate}
\item[(a)]\label{superdiagonal} each $X_i$ has some basis $e_{i,1},\ldots,e_{i,m_i}$ and associated biorthogonal
functionals $f_{i,1},\ldots,f_{i,m_i}$ on $X_i$, such that
\begin{enumerate}
\item[(i)] $f_{i,k}(T(e_{i,k})) = \lambda_i$ for $1 \leq k \leq m_i$;
\item[(ii)] $f_{i,k-1}(T(e_{i,k})) = 1$ for $2 \leq k \leq m_i$;
\item[(iii)] $f_{i,l}(T(e_{i,k})) = 0$ otherwise.
\end{enumerate}
\item[(b)]\label{allevals} if $x$ is an eigenvector of $T$ with eigenvalue $\lambda$, and $|\lambda| \geq
1$, then $x \in X_i$ and $\lambda = \lambda_i$ for some $i \leq n$.
\end{enumerate}

Define $Y = [e_{i,1}]_{1\leq i \leq n, |\lambda_i| = 1} \oplus Z$,
where $[\,\cdot\,]$ denotes norm-closed linear span. To prove the
proposition, it is sufficient to show that
\begin{enumerate}
\item[(I)] $\norm{T^n(x)} \rightarrow \infty$ whenever $x \in X_i\setminus
Y$;
\item[(II)] $\norm{T^n(x)} = \norm{x}$ whenever $x \in X_i \cap Y$
\item[(III)] $\norm{T^n(x)} \rightarrow 0$ whenever $x \in Z$.
\end{enumerate}
First, take $x \in X_i$ such that $f_{i,j}(x) = 0$ for $j > k$.
Using properties (ai)--(aiii), we see that
$$
f_{i,k}(T(x)) \;=\; \sum_{l=1}^{m_i} f_{i,k}(T(e_{i,l}))f_{i,l}(x)
\;=\; \lambda_i f_{i,k}(x)
$$
and $f_{i,j}(T(x)) = 0$ for $j > k$. Thus, by induction,
$f_{i,k}(T^n(x)) = \lambda^n_i f_{i,k}(x)$. Moreover, if $k > 1$
then we calculate
$$
f_{i,k-1}(T(x)) \;=\; \sum_{l=1}^{m_i}
f_{i,k-1}(T(e_{i,l}))f_{i,l}(x) \;=\; \lambda_i f_{i,k-1}(x) +
f_{i,k}(x)
$$
and by a second induction we obtain
$$
f_{i,k-1}(T^n(x)) \;=\; \lambda_i^n f_{i,k-1}(x) +
n\lambda_i^{n-1}f_{i,k}(x).
$$

Now consider (I). If $x \notin Y$ then in particular $f_{i,k}(x)
\neq 0$ for maximal $k \leq m_i$. There are two cases. If
$|\lambda_i| > 1$ then
$$
\norm{f_{i,k}}\norm{T^n(x)} \;\geq\; f_{i,k}(T^n(x)) \;=\;
|\lambda_i|^n |f_{i,k}(x)| \;\rightarrow\; \infty.
$$
If instead $|\lambda_i| = 1$ then we must have $k \geq 2$, thus
$$
\norm{f_{i,k-1}}\norm{T^n(x)} \;\geq\; |f_{i,k-1}(T^n(x))| \;\geq\;
n|f_{i,k}(x)| - |f_{i,k-1}(x)| \;\rightarrow\; \infty.
$$
To see (II), note that the results above give $T^n(e_{i,1}) =
\lambda_i^n e_{i,1}$ straightaway. For (III), observe that by
property (b) above, we have ensured that the restriction $S$ of $T$
to $Z$ has spectral radius $\alpha < 1$. If $x \in Z$ and $\alpha <
\beta < 1$ then
$$
\norm{T^n(x)} \;=\; \norm{S^n(x)} \;\leq\; \norm{S^n}\norm{x}
\;\leq\; \beta^n\norm{x} \;\rightarrow\; 0.
$$
for large enough $n$.

This completes the proof in the complex case. If $X$ is real then we
pass to its complexification $X_\mathbb{C}$ and consider the compact
operator $T_\mathbb{C}$, defined by $T_\mathbb{C}(x + \mi y) = Tx +
\mi Ty$.
\end{proof}

\section{Local estimates}\label{sectionlocalestimates}

Our map $R$ in Theorem \ref{master} is going to be a block diagonal
operator on $X$. In this section, we build the template for the
operators acting on the blocks and gather together some basic
estimates. Let $m,T \in \nat$, $\varepsilon > 0$ and $Y = \lp{p}^T$,
where $4m \leq T$ and $1 \leq p \leq \infty$. Define the operators
$\mapping{S}{Y}{Y}$ and $\mapping{F}{\real}{Y}$ by
$$
S(y) = (y_T,y_1,\ldots,y_{T-1})
$$
where $y = (\oneton{y}{T})$, and
$$
F(a) = (\underbrace{\varepsilon a,\ldots,\varepsilon a}_{m
\;\mathrm{times}},\underbrace{-\varepsilon a,\ldots,-\varepsilon
a}_{m \;\mathrm{times}},0,\ldots,0).
$$
In this way, $S$ can be described as a shift operator and $F$ a
`feed' operator. Let $\mapping{R}{\real\oplus Y}{\real\oplus Y}$ be
defined by $R(a,y) = (a,S(y)+F(a))$. We are interested in the
behaviour of $R^t(a,0)$ at time $t \in \nat$. We can imagine that
$S$ drives an airport baggage carousel and $F$ deposits the
passengers' bags onto the moving belt at a fixed set of positions
(although some of the bags have `negative mass'). The absolute mass
of bags deposited depends on the value of the first coordinate.
Aided by this analogy, we can see that the result of repeated
applications of $R$ to the vector $(a,0)$ can be viewed as the sum
of two bumps: one stationary bump of height $\varepsilon am$ and
base width $2m$, and a moving bump of height $-\varepsilon am$ and
base width again $2m$. The moving bump's motion is periodic, with
period $T$. Let us denote by $P$ the map $(a,y) \mapsto y$.

\begin{lem}
\label{localestimates} Let $1 \leq p \leq \infty$. Firstly, if $m
\leq t \leq T-m$ then
\begin{equation}
\label{lowerest} \norm{PR^t(a,0)} \geq \left\{
\begin{array} {l@{\quad}l}
\left(\frac{2}{p+1}\right)^{p^{-1}}\varepsilon m^{(p+1)p^{-1}} |a| & \mbox{if } p < \infty\\
\varepsilon m|a| & \mbox{if }p=\infty.\\
\end{array} \right.
\end{equation}
Secondly, there exists a constant $L$, depending only on $p$, such
that at all times $t$ we have
\begin{equation}
\label{unifupperest}
\norm{PR^t(a,0)} \leq \left\{
\begin{array} {l@{\quad}l}
L\varepsilon m^{(p+1)p^{-1}} |a| & \mbox{if } p < \infty\\
L\varepsilon m |a| & \mbox{if }p=\infty\\
\end{array} \right.
\end{equation}
and if $t \leq m$ then
\begin{equation}
\label{smalltimeupperest} \norm{PR^t(a,0)} \leq \left\{
\begin{array} {l@{\quad}l}
L\varepsilon m^{p^{-1}}t |a| & \mbox{if } p < \infty\\
L\varepsilon t |a| & \mbox{if }p=\infty.\\
\end{array} \right.
\end{equation}
\end{lem}

\begin{proof}
We estimate the norm of the sum of the standing and moving bumps. If
$p=\infty$ we simply measure the absolute height of the sum of the
bumps to obtain the values listed above, with $L=1$. From now on, we
shall assume that $p < \infty$. Set
$$
L \;=\; \left(\frac{2^{p+3}}{p+1}\right)^{\frac{1}{p}} \;>\;
\left(2+\frac{2^{p+2} + 1}{p+1}\right)^{\frac{1}{p}}.
$$
For (\ref{lowerest}), we have
$$
\norm{PR^t(a,0)}^p \;\geq\; 2\varepsilon^p|a|^p\lint{0}{m}{s^p}{s}
\;=\; \frac{2\varepsilon^p|a|^p}{p+1}m^{p+1}.
$$
To establish (\ref{unifupperest}), we note that the maximum value of
the norm is attained when the supports of the standing and moving
bumps are disjoint, which occurs if and only if $2m \leq t \leq
T-2m$. Thus we estimate
$$
\norm{PR^t(a,0)}^p \;\leq\; 4\varepsilon^p|a|^p\lint{0}{m+1}{s^p}{s}
\;=\; \frac{4\varepsilon^p|a|^p}{p+1}(m+1)^{p+1} \;\leq\;
\frac{2^{p+3}\varepsilon^p|a|^p}{p+1}m^{p+1}.
$$
For (\ref{smalltimeupperest}), when $t \leq m$, we note
\begin{eqnarray*}
\norm{PR^t(a,0)}^p &\leq& 2\varepsilon^p|a|^p\left\{(m-t)t^p +
\lint{0}{t+1}{s^p}{s} + \lint{0}{\frac{t}{2}}{(2s)^p}{s} \right\}\\
&=& 2\varepsilon^p|a|^p\left\{(m-t)t^p + \frac{(t+1)^{p+1}}{p+1} +
\frac{t^{p+1}}{2(p+1)}\right\}\\
&\leq& \left(2 + \frac{2^{p+2} + 1}{p+1} \right)\varepsilon^p mt^p
|a|^p.
\end{eqnarray*}
\end{proof}

In order to build our operator $R$ on a Banach space $X$ with a
symmetric basis, we will need to estimate the norms of certain
vectors in $X$ with reasonable precision. In order to do this, we
combine the estimates of Lemma \ref{localestimates} with a result
closely based on a theorem of Tzafriri \cite{tza:74}. We have
altered the statement of the next result to suit our purposes. In
what follows, $d(\cdot,\cdot)$ indicates Banach-Mazur distance and,
as above, $[\,\cdot\,]$ denotes norm-closed linear span.

\begin{prop}[{\cite[Proposition 5]{tza:74}}]
\label{tzafl2} Let $V$ be a $2^n$-dimensional vector space with
basis $(v_\sigma)_{\sigma \in G}$, where $G$ is the set of all
functions from $\{1,\ldots,n\}$ to $\{-1,1\}$. Suppose that there
are constants $K > 0$ and $r > 2$ such that given scalars
$a_\sigma$, $\sigma \in G$, we have
$$
\frac{K^{-1}}{(2^n)^{\frac{1}{s}}}\left(\sum_{\sigma \in G}
|a_\sigma|^s \right)^{\frac{1}{s}} \leq \norm{\sum_{\sigma \in G}
a_\sigma v_\sigma}\bigg/\norm{\sum_{\sigma \in G} v_\sigma} \leq
\frac{K}{(2^n)^{\frac{1}{r}}}\left(\sum_{\sigma \in G} |a_\sigma|^r
\right)^{\frac{1}{r}}
$$
where $r^{-1} + s^{-1} = 1$. Then there exists $M$, dependent on $K$
and $r$, but independent of $V$ and $n$, with the property that if
we define
$$
z_l = \sum_{\sigma \in G} \sigma(l)v_\sigma
$$
for $1 \leq l \leq n$, then $d([z_l]_{l=1}^n,\lp{2}^n) < M$.
\end{prop}

The proof of the next result closely follows that of \cite[Theorem
1]{tza:74}, although we note that the assumed symmetry of the norm
allows us to bypass the Ramsey arguments which feature in
\cite{tza:74}. Tzafriri's notation has also been modified slightly
to suit our requirements.

\begin{lem}
\label{symmtzafriri} Let $X$ have a normalised symmetric basis
$(e_i)_{i=1}^\infty$ with conjugate system $(e^*_i)_{i=1}^\infty$
and symmetric norm $\normdot$. Then there exists $M > 0$ and $p \in
\{1,2,\infty\}$, a pairwise disjoint family of finite subsets $F_n
\subseteq \nat$, $n \in \nat$, vectors $z_{l,n}$, $1 \leq l \leq n$,
supported on $F_n$ and permutations $\pi_n$ of $F_n$ with three
properties:
\begin{enumerate}
\item \label{simshift} given $n$, if a linear operator $S$ on $X$ satisfies
$S(e_i) = e_{\pi_n(i)}$ for all $i \in F_n$, then $S(z_{l,n}) =
z_{\tau(l),n}$, where $\tau$ is the cycle $(1,\ldots,n)$;
\item \label{uniflp} $d([z_{l,n}]_{l=1}^n,\lp{p}^n) < M$ for all $n$;
\item \label{order} $\pi_n$ has order $n$.
\end{enumerate}
\end{lem}

\begin{proof}
Define
$$
\lambda(n) = \norm{e_1 + \ldots + e_n} \quad\mbox{and}\quad \mu(n) =
\norm{e^*_1 + \ldots + e^*_n}.
$$
We follow the proof of \cite[Theorem 1]{tza:74} in distinguishing
three cases.

Case I: for every $n \in \nat$ there exists $m_n \in \nat$ such that
$\lambda(nm_n)/\lambda(m_n) < 2$. Put $p = \infty$. Set $k_1 = 0$
and, given $k_n$, define $k_{n+1} = k_n + nm_n$. Let
$$
F_n = \{k_n + 1,\ldots,k_n + nm_n\}
$$
and define
$$
z_{l,n} = (e_{k_n + (l-1)m_n + 1} + \ldots e_{k_n +
lm_n})/\lambda(m_n)
$$
for $1 \leq l \leq n$, $n \in \nat$. Finally, define
$$
\pi_n(k_n + (l-1)m_n + r) = \left\{
\begin{array} {l@{\quad}l}
k_n + lm_n + r & \mbox{if } 1 \leq l < n\mbox{ and }1 \leq r \leq m_n\\
k_n + r & \mbox{if }l = n \mbox{ and }1 \leq r \leq m_n.\\
\end{array} \right.
$$
It is clear that the $F_n$ are pairwise disjoint and properties
(\ref{simshift}) and (\ref{order}) hold. Now we prove
(\ref{uniflp}). By the symmetry of the norm, we have $\norm{z_{l,n}}
= 1$. Since
\begin{eqnarray*}
{\textstyle\max_{l=1}^n |a_l|} \;\leq\; \norm{\sum_{l=1}^n
a_lz_{l,n}}
&\leq& {\textstyle\max_{l=1}^n |a_l|}\norm{\sum_{l=1}^n z_{l,n}}\\
&\leq& {\textstyle\max_{l=1}^n
|a_l|}\frac{\lambda(nm_n)}{\lambda(m_n)} \;\leq\; {\textstyle
2\max_{l=1}^n |a_l|}
\end{eqnarray*}
for any scalars $\oneton{a}{n}$, we can see that (\ref{uniflp})
holds for any $M > 2$.

Case II: for every $n \in \nat$ there exists $m_n \in \nat$ such
that $\mu(nm_n)/\mu(m_n) < 2$. Now put $p = 1$ and set $k_n$, $F_n$
and $\pi_n$ exactly as in case I. If we set
$$
z^*_{l,n} = (e^*_{k_n + (l-1)m_n + 1} + \ldots e^*_{k_n +
lm_n})/\mu(m_n).
$$
then we have
$$
{\textstyle\max_{l=1}^n |a_l|} \;\leq\; \norm{\sum_{l=1}^n
a_lz^*_{l,n}} \;\leq\; {\textstyle 2\max_{l=1}^n |a_l|}
$$
just as above. Let $z_{1,n}$ satisfy $\norm{z_{1,n}} = 1$ and
$z^*_{1,n}(z_{1,n}) \geq \frac{1}{2}$, and have support contained in
$\{k_n+1,k_n + m_n\}$, i.e., the support of $z^*_{1,n}$. If we let
$S$ be a linear operator satisfying $S(e_i) = e_{\pi_n(i)}$ for $i
\in F_n$, and define $z_{l,n} = S^{l-1}(z_{1,n})$ for $1 < l \leq
n$, then it follows by the symmetry of the norm that $\norm{z_{l,n}}
= 1$ and $z^*_{l,n}(z_{l,n}) = z^*_{1,n}(z_{1,n})$ whenever $1 \leq
l \leq n$. By design, we have ensured that (\ref{simshift}) holds.
To check (\ref{uniflp}), we observe that
$$
\norm{\sum_{l=1}^n a_lz_{l,n}} \leq \sum_{l=1}^n |a_l| \leq
2\left(\sum_{l=1}^n \lambda_l z^*_{l,n}\right)\left(\sum_{k=1}^n a_k
z_{k,n}\right) \leq 4\norm{\sum_{l=1}^n a_l z_{l,n}}
$$
where $\lambda_l a_l = |a_l|$ for $1 \leq l \leq n$. Therefore
(\ref{uniflp}) holds whenever $M > 4$.

Case III: if neither case I nor case II hold then, following the
proof of \cite[Theorem 1]{tza:74} in case III, we obtain constants
$K > 0$ and $r > 2$ such that for all $n \in \nat$ and scalars
$\oneton{a}{n}$, we have
$$
\frac{K^{-1}}{n^{\frac{1}{s}}}\left(\sum_{l=1}^n |a_l|^s
\right)^{\frac{1}{s}} \leq \frac{1}{\lambda(n)}\norm{\sum_{l=1}^n
a_l e_{n+l}} \leq \frac{K}{n^{\frac{1}{r}}}\left(\sum_{l=1}^n
|a_l|^r \right)^{\frac{1}{r}}
$$
where $r^{-1} + s^{-1} = 1$. We set $p=2$ and
$$
F_n \;=\; \{2^n + 1,\ldots,2^{n+1} \}.
$$
Fix $n$ and let $f$ be a bijection from $F = F_n$ to $G$, where $G$
is as in Proposition \ref{tzafl2}. Put $v_\sigma =
e_{f^{-1}(\sigma)}$ for $\sigma \in G$, and let $z_l$, $1 \leq l
\leq n$, be as in Proposition \ref{tzafl2}. Let $\tau$ be the cycle
$(1,\ldots,n)$, define a permutation $\hat{\pi}$ on $G$ by
$\hat{\pi}(\sigma) = \sigma \circ \tau^{-1}$, and then set $\pi =
f^{-1} \circ \hat{\pi} \circ f$. We have (\ref{order}). If $S$ is an
operator on $X$ satisfying $S(e_i) = e_{\pi(i)}$ then we calculate
\begin{eqnarray*}
S(z_l) \;=\; S\left(\sum_{\sigma \in G} \sigma(l)v_\sigma \right)
&=&
S\left(\sum_{\sigma \in G} \sigma(l)e_{f^{-1}(\sigma)} \right)\\
&=& \sum_{\sigma \in G} \sigma(l)e_{f^{-1}(\hat{\pi}(\sigma))}\\
&=& \sum_{\sigma \in G} \sigma(l)v_{\hat{\pi}(\sigma)} \;=\;
\sum_{\sigma \in G} (\sigma \circ \tau)(l)v_\sigma \;=\;
z_{\tau(l)}.
\end{eqnarray*}
Moreover, by construction, we have ensured that
$d([z_l]_{l=1}^n,\lp{p}^n) < M$.
\end{proof}

We remark that we can follow the proof of \cite[Theorem 1]{tza:74} a
little more to show that the subspaces $[z_{l,n}]_{l=1}^n$, $n \in
\nat$, are uniformly complemented in $X$, that is, they are the
images of a sequence of projections which are uniformly bounded in
norm. However, we do not require this particular property of the
$[z_{l,n}]_{l=1}^n$.

\section{Proof of Theorem \ref{master}}\label{sectionmainproof}

Let $X$, $Y$, $E$ and $J$ be as in Theorem \ref{master}. As $X$ has
a symmetric basis, it is isomorphic to its closed,
finite-codimensional subspaces. Moreover, $X$ is isomorphic to the
space $X^{d-1}_2$, which denotes the product $X^{d-1}$, endowed with
the norm given by $\norm{(\oneton{x}{d-1})}^2 = \sum_{j=1}^{d-1}
\norm{x_j}^2$. Therefore, by considering a suitable isomorphism,
Theorem \ref{master} follows from the following proposition.

\begin{prop}\label{realmaster}
Whenever $E$ is a closed, symmetric subset of $\sph{\ell^d_2}$, with
$2 \leq d < \infty$, then there exists an operator
$\mapping{R}{\ell^d_2 \oplus X^{d-1}_2}{\ell^d_2 \oplus X^{d-1}_2}$
with two properties:
\begin{enumerate}
\item if $u \in E$ then $\norm{R^n(u,x)} \rightarrow \infty$;
\item if $u \in \sph{\ell^d_2}\setminus E$ then there is a subsequence $(R^{n_i}(u,x))$
of $(R^n(u,x))$ such that $R^{n_i}(u,x) \rightarrow (u,x)$ weakly.
\end{enumerate}
\end{prop}

We shall prove Proposition \ref{realmaster} with a sequence of
lemmas. The proofs in the real and complex cases are practically
identical. First, we consider a map $\rho$ defined on
$\sph{\ell^d_2}$ by
$$
\rho(u,v)^2 \;=\; 1 - |\ip{u}{v}|^2,
$$
where $\ip{\cdot}{\cdot}$ denotes the usual inner product. This
function $\rho$ is a pseudometric on $\sph{\ell^d_2}$ which, given
$u \in \sph{\ell^d_2}$, identifies the points $\lambda u$,
$|\lambda| = 1$. We shall also define
$$
\rho(u,E) \;=\; \inf\setcomp{\rho(u,v)}{v \in E}.
$$
Since $E$ is closed and symmetric, it follows that $\rho(u,E) = 0$
if and only if $u \in E$. Given $v \in \sph{\ell^d_2}\setminus E$,
we select an orthonormal basis $e_{v,1},\ldots,e_{v,d-1}$ of the
perpendicular subspace $v^\perp$ and define
$\mapping{\Delta_v}{\ell^d_2}{\ell^{d-1}_2}$ by
$$
\Delta_v(u) \;=\;
\frac{1}{\rho(v,E)}(\ip{u}{e_{v,1}},\ldots,\ip{u}{e_{v,d-1}}).
$$
Obviously, $\pnorm{\Delta_v(u)}{2} = \rho(u,v)/\rho(v,E)$ whenever
$u \in \sph{\ell^d_2}$, where $\pnormdot{2}$ denotes the usual norm
on $\ell^{d-1}_2$.

Let $W_n = \setcomp{v \in \sph{\ell^d_2}}{\rho(v,E) \geq 2^{-n}}$.
It is a straightforward matter to show that for each $n$, we can
find a $n^{-1}$-net of $\sph{\ell^d_2}$, with respect to $\rho$,
which has size of order $n^{2d-1}$ (or $n^{d-1}$ if we are
considering real Banach spaces). Therefore, there exists an integer
$K$ such that, for each $n$, there exist vectors
$$
v^n_1,\ldots v^n_{K2^{n(2d-1)}} \in W_n,
$$
with repetitions if necessary, with the property that for any $u \in
W_n$, we can find $v^n_i$ satisfying $\rho(u,v^n_i) \leq 2^{-n}$.

\begin{lem}\label{inequalities2}
Let $u \in \sph{\ell^d_2}$. Firstly
\begin{equation}\label{lambdaupperest}
\pnorm{\Delta_v(u)}{2} \leq 2^n \quad\mbox{whenever }v \in W_n.
\end{equation}
Secondly, if $u \in E$ then
\begin{equation}\label{largeapprox}
\pnorm{\Delta_v(u)}{2} \geq 1 \quad\mbox{whenever }v \in
\sph{\ell^d_2}\setminus E.
\end{equation}
Finally, if $u \not\in E$ then there exists $n_0$ with the property
that whenever $n > n_0$, there exists $i$ in the range $1 \leq i
\leq K2^{n(2d-1)}$, such that
\begin{equation}\label{lambdaapprox2}
\pnorm{\Delta_{v^n_i}(u)}{2} \;\leq\; 2^{n_0+1-n}.
\end{equation}
\end{lem}

\begin{proof}
We only prove (\ref{lambdaapprox2}). Fix $n_0$ such that $u \in
W_{n_0}$. For $n
> n_0$, we can find $v = v^n_i$ such that $\rho(u,v) \leq 2^{-n}$.
Therefore
$$
\rho(v,E) \;\geq\; \rho(u,E) - \rho(u,v) \;\geq\; 2^{-n_0} - 2^{-n}
\;\geq\; 2^{-(n_0+1)}
$$
and $\pnorm{\Delta_v(u)}{2} \leq 2^{n_0+1-n}$.
\end{proof}

We take constants $m_k,T_k \in \nat$ and $\varepsilon_k
> 0$. The values of these constants will be chosen in due course.
Let $M$, $p$, $F_n$, $z_{l,n}$ and $\pi_n$ be as in Lemma
\ref{symmtzafriri}. Define
$$
S(e_i) = \left\{
\begin{array} {l@{\quad}l}
e_{\pi_{T_k}(i)} & \mbox{if } i \in F_{T_k} \mbox{ for some }T_k\\
e_i & \mbox{otherwise}\\
\end{array} \right.
$$
and extend S linearly to $X$. As $\normdot$ is symmetric, $S$ is an
isometry. Define operators
$\mapping{S_k}{[z_{l,T_k}]_{l=1}^{T_k}}{[z_{l,T_k}]_{l=1}^{T_k}}$
and $\mapping{F_k}{\real}{[z_{l,T_k}]_{l=1}^{T_k}}$ by
$$
S_k\left(\sum_{l=1}^{T_k} y_l z_{l,T_k}\right) = \sum_{l=1}^{T_k}
y_l z_{\tau(l),T_k}
$$
where $\tau$ is the cycle $(1,\ldots,T_k)$, and
$$
F_k(a) = a\varepsilon_k\sum_{l=1}^{m_k}z_{l,T_k} -
a\varepsilon_k\sum_{l=m_k+1}^{2m_k}z_{l,T_k}.
$$
Then define $R_k$ on $\real\oplus [z_{l,T_k}]_{l=1}^{T_k}$ by
$$
R_k(a,y) = (a,S_k(y) + F_k(a))
$$
and let $P_k(a,y) = y$ for $a \in \real$ and $y \in
[z_{l,T_k}]_{l=1}^{T_k}$. Let $Q$ and $Q_j$ be the standard
projections of $\ell^d_2 \oplus X^{d-1}_2$ onto $\ell^d_2$ and onto
the $j$th copy of $X$, respectively. We define integers $C_1 = 1$
and $C_{n+1} = C_n + K2^{n(2d-1)}$ for $n \geq 1$, and set $w_k =
v^n_{k+1-C_n}$ whenever $C_n \leq k < C_{n+1}$.

Finally, if $u \in \ell^d_2$ and $x \in X^{d-1}_2$, we can define an
operator $R$ on $\ell^d_2 \oplus X^{d-1}_2$ by
$$
QR(u,x) \;=\; u \quad\mbox{and}\quad Q_jR(u,x) = SQ_j(0,x) +
\sum_{k=1}^\infty \frac{F_k(\ip{u}{e_{w_k,j}})}{\rho(w_k,E)}
$$
for $1 \leq j \leq d-1$ and where $e_{w_k,j}$ is the $j$th element
of the given basis of $w_k^\perp$, chosen above.

Of course, it is necessary to choose the constants $m_k,T_k$ and
$\varepsilon_k$ so that $R$ is bounded and maps into $X$. First let
$m_1$ = 1 and $T_0 = 1$. Then set $T_k = (5^{dn} + 1)T_{k-1}$, $m_k
= T_{k-1} - m_{k-1}$ and
$$
\varepsilon_k = \left\{
\begin{array} {l@{\quad}l}
\frac{n}{m_k^{(p+1)p^{-1}}} & \mbox{if } p=1 \mbox{ or }p=2\\
\frac{n}{m_k} & \mbox{if } p = \infty
\end{array} \right.
$$
whenever $C_n \leq k < C_{n+1}$. Our first task is to show that,
with respect to these constants, $R$ is a bounded operator mapping
into $\ell^d_2 \oplus X^{d-1}_2$.

\begin{lem}
\label{rgoodop} The operator $R$ is bounded and maps into $\ell^d_2
\oplus X^{d-1}_2$.
\end{lem}

\begin{proof}
It is enough to show that $\sum_{k=1}^\infty
\rho(w_k,E)^{-1}F_k(\ip{u}{e_{w_k,j}})$ is absolutely summable for
$1 \leq j \leq d-1$. Assume $u \in \sph{\ell^d_2}$. By Lemma
\ref{symmtzafriri}, part \ref{uniflp}, we have
\begin{equation}
\label{equiv} M^{-\frac{1}{2}}\pnorm{y}{p} \leq
\norm{\sum_{l=1}^{T_k} y_l z_{l,T_k}} \leq
M^{\frac{1}{2}}\pnorm{y}{p}
\end{equation}
where $y = (y_1,\ldots,y_{T_k}) \in \lp{p}^{T_k}$ and $\pnormdot{p}$
is the usual norm on $\lp{p}^{T_k}$. We shall assume that $M \geq
L$, where $L$ is as in Lemma \ref{localestimates}. Note that $F_k(a)
= S_k(0) + F_k(a) = P_kR_k(a,0)$. Therefore, from
(\ref{smalltimeupperest}) with $t = 1$, (\ref{lambdaupperest}),
(\ref{equiv}) and the definition of $\varepsilon_k$, we have
\begin{eqnarray*}
\frac{\norm{F_k(\ip{u}{e_{w_k,j}})}}{\rho(w_k,E)} &\leq& \left\{
\begin{array} {l@{\quad}l}
M^{\frac{3}{2}}\varepsilon_km_k^{p^{-1}}\rho(w_k,E)^{-1}|\ip{u}{e_{w_k,j}}| & \mbox{if } p=1,2\\
M^{\frac{3}{2}}\varepsilon_k\rho(w_k,E)^{-1}|\ip{u}{e_{w_k,j}}| & \mbox{if }p=\infty\\
\end{array} \right.\\
&\leq& M^{\frac{3}{2}}n \pnorm{\Delta_{w_k}(u)}{2} m_k^{-1}\\
&\leq& M^{\frac{3}{2}}n 2^n m_k^{-1}
\end{eqnarray*}
whenever $C_n \leq k < C_{n+1}$.

From the definitions of $m_k$ and $T_k$, we obtain
\begin{equation}
\label{mkTk} m_{k+1} = T_k - m_k = T_k - T_{k-1} + m_{k-1} \geq
5^{dn}T_{k-1} \geq 5^{dn}m_k
\end{equation}
whenever $C_n \leq k < C_{n+1}$. In particular, $m_{k+1} \geq 5^d
m_k$. Therefore
\begin{eqnarray*}
\sum_{k=1}^\infty \frac{\norm{F_k(\ip{u}{e_{w_k,j}})}}{\rho(w_k,E)}
&\leq& \sum_{n=1}^\infty\sum_{k = C_n}^{C_{n+1}-1}
M^{\frac{3}{2}}n2^n m_k^{-1}\\
&\leq& M^{\frac{3}{2}}m_1^{-1} \sum_{n=1}^\infty\sum_{k =
C_n}^{C_{n+1}-1}
\frac{n2^n}{5^{d(k-1)}}\\
&\leq& M^{\frac{3}{2}}5^d \sum_{n=1}^\infty\sum_{k =
C_n}^{C_{n+1}-1}
\frac{n2^n}{5^{dn}}\\
&=& {\textstyle M^{\frac{3}{2}}}5^dK \sum_{n=1}^\infty 2^{n(2d-1)}
\frac{n2^n}{5^{dn}}\;=\; {\textstyle M^{\frac{3}{2}}}5^dK
\sum_{n=1}^\infty n{\textstyle(\frac{4}{5})^{dn}}
\end{eqnarray*}
bearing in mind that $k-1 \geq C_n-1 \geq n-1$ whenever $C_n \leq k
< C_{n+1}$. Hence $R$ is bounded and $R(u,x) \in \ell^d_2 \oplus
X^{d-1}_2$.
\end{proof}

In order to analyse the behaviour of $R^m(u,x)$, it will help to
consider separately $R^m(u,0)$ and $R^m(0,x)$.

\begin{lem}
\label{splitproblem} Given $(u,x) \in \lp{2}^d \oplus X^{d-1}_2$, we
have
\begin{equation}
\label{shift} Q_jR^m(0,x) = S^mQ_j(0,x)
\end{equation}
and
\begin{equation}
\label{perturb} Q_jR^m(u,0) = \sum_{k=1}^\infty \frac{P_k
R^m_k(\ip{u}{e_{w_k,j}},0)}{\rho(w_k,E)}
\end{equation}
for $1 \leq j \leq d-1$ and for all $m$.
\end{lem}

\begin{proof}
We proceed by induction. Clearly $Q_jR(0,x) = SQ_j(0,x)$, so
$(\ref{shift})$ holds for $m=1$. If (\ref{shift}) holds for some $m
\geq 1$ and $R^m(0,x) = (0,y)$ then
$$
Q_jR^{m+1}(0,x) = Q_jR(0,y) = SQ_j(0,y) = SQ_jR^m(0,x) =
S^{m+1}Q_j(0,x).
$$
Now
$$
P_kR_k(a,0) = S_k(0) + F_k(a) = F_k(a)
$$
and $SQ_j(u,0) = 0$, so (\ref{perturb}) holds for $m=1$. Assume that
$(\ref{perturb})$ holds for some $m \geq 1$. Suppose that
$$
P_k R^m_k(a,0) = y = \sum_{l=1}^{T_k}y_l z_{l,T_k}.
$$
By Lemma \ref{symmtzafriri}, we have ensured that $S(y) = S_k(y)$.
Furthermore, we observe
\begin{eqnarray*}
P_k R^{m+1}_k(a,0) \;=\; P_k R_k(R^m_k(a,0)) &=& P_k R_k(a,y)\\
&=& S_k(y) + F_k(a)\\
&=& S(y) + F_k(a) \;=\; SP_kR^m_k(a,0) + F_k(a).
\end{eqnarray*}
Therefore, if $R^m(u,0) = (u,z)$ then
\begin{eqnarray*}
Q_jR^{m+1}(u,0) &=& Q_jR(u,z)\\
&=& SQ_j(0,z) + \sum_{k=1}^\infty \frac{F_k(\ip{u}{e_{w_k,j}})}{\rho(w_k,E)}\\
&=& SQ_jR^m(u,0) + \sum_{k=1}^\infty \frac{F_k(\ip{u}{e_{w_k,j}})}{\rho(w_k,E)}\\
&=& S\left(\sum_{k=1}^\infty \frac{P_k R^m_k(\ip{u}{e_{w_k,j}},0)}{\rho(w_k,E)}\right) + \sum_{k=1}^\infty \frac{F_k(\ip{u}{e_{w_k,j}})}{\rho(w_k,E)}\\
&=& \sum_{k=1}^\infty \frac{SP_k R^m_k(\ip{u}{e_{w_k,j}},0) + F_k(\ip{u}{e_{w_k,j}})}{\rho(w_k,E)}\\
&=& \sum_{k=1}^\infty \frac{P_k R^{m+1}_k(\ip{u}{e_{w_k,j}},0)}{\rho(w_k,E)}
\end{eqnarray*}
as required.
\end{proof}

The consequence of Lemma \ref{splitproblem} is that we can split the
analysis of $R^m (u,x)$ into two parts: the `shift' and the
`perturbation'. First, we examine the behaviour of the shift.

\begin{lem}
\label{shiftconvergence} Given $x \in X^{d-1}_2$, we have
$\norm{R^m(0,x)} = \norm{(0,x)}$ for all $m$. Moreover,
$R^{T_k}(0,x) \stackrel{w}{\rightarrow} (0,x)$.
\end{lem}

\begin{proof}
Given (\ref{shift}) and the fact that $S$ is an isometry, the first
assertion is trivial. Now consider the weak convergence. Let $f \in
\dual{X}$ with $\norm{f} = 1$, $\varepsilon > 0$ and $1 \leq j \leq
d-1$. If $Q_j(0,x) = \sum_{i=1}^\infty x_i e_i$, we take $k \in
\nat$ such that
$$
\norm{\sum_{l=k+1}^\infty\sum_{i \in F_{T_l}} x_i e_i} <
\varepsilon.
$$
Since $T_l$ divides $T_r$ whenever $l \leq r$, we can see that
$\pi_{T_l}^{T_r}$ is the identity for such $l$. Therefore, if $r
\geq k$, we estimate
\begin{eqnarray*}
|f(Q_j R^{T_k}(0,x) - Q_j(0,x))| &=& |f(S^{T_r}Q_j(0,x) - Q_j(0,x))|\\
&=& \left| f\left(\sum_{l=r+1}^\infty\sum_{i \in
F_{T_l}}x_ie_{\pi^{T_r}_{T_l}(i)} -
\sum_{l=r+1}^\infty\sum_{i \in F_{T_l}} x_i e_i\right)\right|\\
&\leq& 2\norm{\sum_{l=r+1}^\infty\sum_{i \in F_{T_l}} x_i e_i}\\
&\leq& 2\norm{\sum_{l=k+1}^\infty\sum_{i \in F_{T_l}} x_i e_i} \;<\;
2\varepsilon
\end{eqnarray*}
by symmetry of the norm. Thus $Q_jR^{T_k}(0,x)
\stackrel{w}{\rightarrow} Q_j(0,x)$ whenever $1 \leq j \leq d-1$.
The weak convergence of $R^{T_k}(0,x)$ to $(0,x)$ follows.
\end{proof}

Now we analyse the behaviour of the perturbation. Ultimately, it is
the perturbation that drives the behaviour of the system as a whole.

\begin{lem}
\label{perturbconvergence} If $u \in E$ then
$$
\norm{R^m(u,0)} \rightarrow \infty.
$$
On the other hand, if $u \in \sph{\ell^d_2}\setminus E$ then there
exists $k_n$ in the range $C_n \leq k_n < C_{n+1}$ with the property
that
$$
\norm{R^{T_{k_n - 1}}(u,0) - (u,0)} \rightarrow 0.
$$
\end{lem}

\begin{proof}
Let $u \in E$ and suppose that $p=1$ or $p=2$. Assume that $m_k \leq
m < T_k - m_k = m_{k+1}$ and $C_n \leq k < C_{n+1}$. Then by
(\ref{lowerest}), (\ref{largeapprox}), (\ref{equiv}),
(\ref{perturb}) and the definition of $\varepsilon_k$, we have
\begin{eqnarray*}
\norm{R^m(u,0)}^2 &\geq& \sum_{j=1}^{d-1}\norm{Q_jR^m(u,0)}^2\\
&\geq& \sum_{j=1}^{d-1}\left(\frac{\norm{P_k
R^m_k(\ip{u}{e_{w_k,j}},0)}}{\rho(w_k,E)}\right)^2\\
&\geq&
\sum_{j=1}^{d-1}M^{-1}\left(\frac{2}{p+1}\right)^{2p^{-1}}\varepsilon^2_km_k^{2(p+1)p^{-1}}\left(\frac{|\ip{u}{e_{w_k,j}}|}{\rho(w_k,E)}\right)^2\\
&=&
M^{-1}\left(\frac{2}{p+1}\right)^{2p^{-1}}n^2\sum_{j=1}^{d-1}\left(\frac{|\ip{u}{e_{w_k,j}}|}{\rho(w_k,E)}\right)^2\\
&\geq& M^{-1}\left(\frac{2}{p+1}\right)^{2p^{-1}}n^2
\end{eqnarray*}
If $u \in E$ and $p = \infty$ then similarly, we obtain
$$
\norm{R^m(u,0)}^2 \;\geq\; M^{-1}n^2.
$$
Either way, $\norm{R^m(u,0)} \rightarrow \infty$.

Instead, if $u \in \sph{\ell^d_2}\setminus E$ then by
(\ref{lambdaapprox2}), for every $n > n_0$ there exists $k_n$ in the
range $C_n \leq k_n < C_{n+1}$, such that
$$
\pnorm{\Delta_{w_{k_n}}(u)}{2} \;\leq\; 2^{n_0+1-n}
$$
By (\ref{unifupperest}), (\ref{equiv}) and the definition of
$\varepsilon_k$, we have
\begin{eqnarray}
\nonumber \norm{P_{k_n}
R^{T_{k_n-1}}_{k_n}(\ip{u}{e_{w_{k_n},j}},0)} &\leq& \left\{
\begin{array} {l@{\quad}l}
M^{\frac{3}{2}}\varepsilon_{k_n}m_{k_n}^{(p+1)p^{-1}}|\ip{u}{e_{w_{k_n},j}}| & \mbox{if } p = 1,2\\
M^{\frac{3}{2}}\varepsilon_{k_n}m_{k_n}|\ip{u}{e_{w_{k_n},j}}| & \mbox{if }p = \infty\\
\end{array} \right.\\
\label{Tk}&\leq& M^{\frac{3}{2}}\rho(w_{k_n},E)n2^{n_0+1-n}.
\end{eqnarray}
Then we notice that if $r \leq k_n - 1$, we have
\begin{eqnarray}
\label{lessTk}\norm{P_rR^{T_{k_n-1}}_r(\ip{u}{e_{w_r,j}},0)} &=& 0
\end{eqnarray}
because $R^{T_r}_r$ is the identity and $T_r$ divides $T_{k_n-1}$
whenever $r \leq k_n-1$. Now we have to estimate
$\norm{P_rR^{T_{k_n-1}}_r(\ip{u}{e_{w_r,j}},0)}$ for $r \geq
k_n+1$. If $r \geq k_n+1$ then from (\ref{mkTk}), we have
$$
m_r \geq 5^{d(r-(k_n+1))}m_{k_n+1} \geq
5^{d(r-(k_n+1))}5^{dn}T_{k_n-1}.
$$
Take $l \geq n$ such that $C_l \leq r < C_{l+1}$. Again we assume
that $M \geq L$, with $L$ as in Lemma \ref{localestimates}. We apply
(\ref{smalltimeupperest}), (\ref{lambdaupperest}) and (\ref{equiv})
to obtain
\begin{eqnarray}
\label{moreTk} &&\norm{P_rR^{T_{k_n-1}}_r(\ip{u}{e_{w_r,j}},0)}\\
\nonumber &\leq& \left\{\begin{array} {l@{\quad}l}
M^{\frac{3}{2}}\varepsilon_r m_r^{p^{-1}}T_{k_n-1}|\ip{u}{e_{w_r,j}}| & \mbox{if } p=1,2\\
M^{\frac{3}{2}}\varepsilon_r T_{k_n-1}|\ip{u}{e_{w_r,j}}| & \mbox{if }p=\infty\\
\end{array} \right.\\
\nonumber &\leq& M^{\frac{3}{2}} T_{k_n-1}\frac{\rho(w_r,E)l2^l}{m_r}\\
\nonumber &\leq& M^{\frac{3}{2}}\rho(w_r,E)\frac{l2^l}{5^{d(r-(k_n+1))}5^{dn}}\\
\nonumber &\leq&
M^{\frac{3}{2}}\rho(w_r,E)\frac{l2^l}{5^{dn}5^{d(l-n)}}
\quad\mbox{since }l-n \leq r-(k_n+1)\\
\nonumber &=& M^{\frac{3}{2}}\rho(w_r,E)\frac{l2^l}{5^{dl}}
\end{eqnarray}
Combining (\ref{perturb}) with (\ref{Tk}), (\ref{lessTk}) and
(\ref{moreTk}) gives
\begin{eqnarray*}
\norm{Q_jR^{T_{k_n-1}}(u,0)} &\leq& \sum_{r=1}^\infty
\frac{\norm{P_rR^{T_{k_n-1}}_r(\ip{u}{e_{w_r,j}},0)}}{\rho(w_r,E)}\\
&=& \sum_{r=k_n}^\infty \frac{\norm{P_rR^{T_{k_n-1}}_r(\ip{u}{e_{w_r,j}},0)}}{\rho(w_r,E)}\\
&=& M^{\frac{3}{2}}n2^{n_0+1-n} +
\sum_{r=k_n+1}^\infty
\frac{\norm{P_rR^{T_{k_n-1}}_r(\ip{u}{e_{w_r,j}},0)}}{\rho(w_r,E)}\\
&\leq& M^{\frac{3}{2}}n2^{n_0+1-n} +
M^{\frac{3}{2}}\sum_{l=n}^\infty K2^{l(2d-1)} \frac{l2^l}{5^{dl}}\\
&=& M^{\frac{3}{2}}n2^{n_0+1-n} +
M^{\frac{3}{2}}K\sum_{l=n}^\infty l({\textstyle \frac{4}{5}})^{dl}\\
&\rightarrow& 0
\end{eqnarray*}
as $n \rightarrow \infty$. This concludes the proof.
\end{proof}

\begin{proof}[Proof of Proposition \ref{realmaster}]
Let $u \in E$. By Lemmas \ref{shiftconvergence} and
\ref{perturbconvergence} we have
$$
\norm{R^m (u,x)} \geq \norm{R^m (u,0)} - \norm{R^m(0,x)} = \norm{R^m
(u,0)} - \norm{(0,x)} \rightarrow \infty
$$
as $m \rightarrow \infty$.

Now suppose $u \in \sph{\ell^d_2}\setminus E$. Again by Lemmas
\ref{shiftconvergence} and \ref{perturbconvergence}, we can pick
suitable $k_n$ such that
$$
R^{T_{k_n-1}}(u,x) = R^{T_{k_n-1}}(u,0) + R^{T_{k_n-1}}(0,x)
\stackrel{w}{\rightarrow} (u,0) + (0,x) = (u,x).
$$
\end{proof}

If $X = \czero$ or $X = \lp{p}$, $1 \leq p < \infty$, then we can
simplify the proof of Theorem \ref{master} by replacing the
$z_{l,n}$ with unit vectors and replacing the corresponding $\pi_n$
with cycles. Since there is a Banach space with a symmetric basis
but containing no isomorphic copy of $\czero$ or $\lp{p}$, $1 \leq p
< \infty$, \cite{fj:74}, it is not possible to obtain Theorem
\ref{master} by proving it in the cases $X = \czero$ and $X =
\lp{p}$, and then appealing to complemented subspaces.

\section{Problems}

Since the operators constructed in this note rely fundamentally on
permutations of basis vectors, it makes sense to pose the following
question.

\begin{prob}
If $X$ is a Banach space with an unconditional basis, does there
exist an operator $T$ on $X$ with the property that $\norm{T^n(x)}
\rightarrow \infty$ for some $x \in X$, and $\norm{T^n y}
\not\rightarrow \infty$ for all $y$ in a non-empty, open subset of
$X$?
\end{prob}

Also, since the operators which feature above cannot be compact, by
Proposition \ref{compact}, the next question seems natural to us.

\begin{prob}
Is it possible to find a sum $I+T$, where $T$ is compact, which
satisfies the properties given in the abstract? In particular, does
the Argyros-Haydon space \cite{ah:09} admit such an operator?
\end{prob}

If no sum $I+T$, $T$ compact, satisfies the properties given in the
abstract, then this suggests to us that some kind of unconditional
structure is necessary in order to construct such operators.

Finally, we make a remark about the title of this note. The operator
$R$ constructed above can be viewed as a machine which acts on a
countable family of disjoint cycles. This family of disjoint cycles
can be seen as a countable directed graph. We speculate that it may
be possible to construct operators with other interesting properties
by basing them on more complicated directed graphs.

\bibliographystyle{amsplain}

\end{document}